\documentclass[12pt]{amsart}
\usepackage{amssymb,a4wide,xy}
\usepackage{amsmath,amscd,amsgen,amsmath,amsthm,amstext,amsbsy,amsopn,amsfonts,latexsym,enumerate}
\usepackage[latin1]{inputenc}
\xyoption{all}

\newcommand{\vf}{\varphi}

\newcommand{\bC}{\mathbb{C}}

\newcommand{\bB}{\mathbb{B}}
\newcommand{\cB}{\mathfrak{B}}
\newcommand{\bE}{\mathbb{E}}
\newcommand{\bJ}{\mathbb{J}}

\newcommand{\Om}{\Omega}
\newcommand{\om}{\omega}
\newcommand{\aaa}{{\bf{\mathfrak{a}}}}
\newcommand{\bbb}{{\bf b}}

\newcommand{\lra}{\longrightarrow}
\newcommand{\ol}{\overline}

\newtheorem{theo}{Theorem}[section]
\newtheorem{Lem}[theo]{Lemma}
\newtheorem{Prop}[theo]{Proposition}
\newtheorem{coro}[theo]{Corollary}

\theoremstyle{definition}
\newtheorem{defi}[theo]{Definition}
\newtheorem{ex}{Example}[section]
\theoremstyle{remark}

\newtheorem{axiom}{Axiom}

\DeclareMathOperator{\ann}{Ann} \DeclareMathOperator{\lann}{lAnn}
\DeclareMathOperator{\rann}{rAnn} \DeclareMathOperator{\act}{Act}
\DeclareMathOperator{\alt}{Alt} \DeclareMathOperator{\galt}{gAlt}
\DeclareMathOperator{\asoci}{Asoci}
\DeclareMathOperator{\soci}{Soci} \DeclareMathOperator{\bim}{Bim}
\DeclareMathOperator{\ch}{char} \DeclareMathOperator{\Ker}{Ker}

\numberwithin{equation}{section}

\begin{document}

\title{Actor of an alternative algebra}

\author[J.M. Casas]{J.M. ~Casas}
\address{\small \rm Jos\'e Manuel Casas: Departamento de Matem\'atica Aplicada I, Universidad de Vigo, Pontevedra, E-36005,  Spain}
\email{jmcasas@uvigo.es}
\author[T. Datuashvili]{T. ~Datuashvili}
\address{\small \rm Tamar Datuashvili: A. Razmadze Math. Inst., Alexidze str.1, 0193 Tbilisi, Georgia}
\email{tamar@rmi.acnet.ge}
\author[M. Ladra]{M. ~Ladra}
\address{\small \rm  Manuel Ladra: Departamento de \'Algebra, Universidad de Santiago de Compostela\\15782
Santiago de Compostela, Spain}
\email{manuel.ladra@usc.es}

\subjclass{08C05, 17D05}

\keywords{alternative algebra,  action, universal
strict general actor,  actor,  category of interest}

\begin{abstract}
We define a category $\galt$ of g-alternative algebras over a
field $F$ and present the category of alternative algebras $\alt$
as a full subcategory of $\galt$; in the case $\ch F\neq 2$, we
have $\alt=\galt$. For any g-alternative algebra $A$ we give a
construction of a universal strict general actor $\cB(A)$ of $A$.
We define the subset $\asoci(A)$ of $A$, and show that it is a
$\cB(A)$-substructure of $A$. We prove that if $\asoci(A)=0$, then
there exists an actor of $A$ in $\galt$ and $\act(A)=\cB(A)$. In
particular, we obtain that if $A$ is anticommutative and
$\ann(A)=0$, then there exists an actor of $A$ in $\galt$; from
this, under the same conditions, we deduce the existence of an
actor in $\alt$.
\end{abstract}

\maketitle

\section{Introduction}
 The paper is dedicated to the problem of the existence of
 universal acting objects - actors in the category of alternative algebras (see Section \ref{def}
for the definitions).  Actions in algebraic categories were
studied in \cite{Ho}, \cite{Mc}, \cite{LS}, \cite{Lu}, \cite{No},
\cite{Lo}, \cite{LL} and in other papers. The authors were looking
for the analogs of the group of automorphisms of a group for
associative algebras, rings, commutative associative algebras, Lie
algebras, crossed modules and Leibniz algebras. They gave the
constructions of universal objects (universal in different senses)
in the corresponding categories and studied their properties.
These objects are: the ring of bimultiplications of a ring, the
associative algebras of multiplications (or bimultipliers) of an
associative algebra and multiplications (or multipliers) of a
commutative associative algebra, the actor of a crossed module,
the Lie algebra of derivations and the Leibniz algebra of
biderivations of a Lie and a Leibniz algebras, respectively. In
\cite{BJ} was proposed a categorical approach to this problem,
which was continued in \cite{BJK}, \cite{BJK1}, \cite{BBJ},
\cite{Bou} and \cite{BB}. In particular, representable internal
object action was defined and necessary and sufficient conditions
of its existence in a semi-abelian category was established. In
\cite{CDL2} for any category of interest $\bC$, in the sense of
\cite{Orz}, we defined the corresponding category of groups with
operations $\bC_G$, $\bC\subseteq \bC_G$; for any object $A\in
\bC$ we defined an actor $\act(A)$ of $A$; this notion is
equivalent to the one of split extension classifier for $A$,
defined in a semi-abelian category in \cite{BJK}. In the same
paper we gave a definition and a construction of a universal
strict general actor $\cB(A)$ of $A$, which is an object in
$\bC_G$ in general and has all universal properties of the objects
listed above. We proved that there exists an actor of $A$ in $\bC$
if and only if the semidirect product $\cB(A)\ltimes A$ is an
object in $\bC$ and, in this case, $\cB(A)$ is an actor of $A$ in
$\bC$. Applying this result we considered examples of groups,
associative algebras, Lie and Leibniz algebras, crossed modules
and precrossed modules. In \cite{CDL3} we gave the construction of
an actor of a precrossed module, where we introduced the notion of
a generalized Whitehead's derivation. In the present paper we
define a category of general alternative (g-alternative) algebras
over a field $F$ (denoted by $\galt$), which is a category of
interest. We present the category of alternative algebras (denoted
by $\alt$) as a full subcategory of $\galt$. In the case, where
$\ch F \neq 2$, we have $\alt=\galt$. Applying the results of
\cite{CDL2}, for any g-alternative algebra $A$ we give a
construction of a universal strict general actor of $A$ and obtain
sufficient conditions for the existence of an actor of $A$ in
$\galt$. From this we easily deduce analogous results for the case
of alternative algebras.

In Section \ref{def} we recall the definitions of a category of
interest, the general category of groups with operations of a
category of interest, and give examples. We recall the definitions
of structure, derived action and actor in a category of interest.
We state a necessary and sufficient condition for the existence of
an actor in a category of interest in terms of a universal strict
general actor. The construction of this object for the case of the
category $\galt$ is given in Section \ref{acting}. In the
beginning of this section we define derived actions in the
categories of g-alternative and alternative algebras. We state
some properties of g-alternative algebras in terms of the axioms
of the corresponding category of interest, which are essentially
well-known for the case of alternative algebras. Then we give a
construction of a universal strict general actor $\cB(A)$ of a
g-alternative algebra $A$; as special cases of known definitions
we define a semidirect product and the center of an object in
$\galt$ (equivalently, in $\alt$). In Section \ref{properties} for
any g-alternative algebra $A$ we define the $\cB(A)$-substructure
$\soci(A)$ of $A$, due to which we define certain subset
$\asoci(A)$, which turned out to be a $\cB(A)$-substructure of
$A$. We study the properties of this object, which are applied in
what follows. In Section \ref{sufficient} we give an example of
algebra, which shows that the object $\cB(A)$ is not a
g-alternative algebra in general; we state sufficient conditions
for $\cB(A)\in \galt$. In Section \ref{existence_galt} we give
examples of algebras for which the action of $\cB(A)$ on $A$ is
not a derived action in general. The final result states that if
$\asoci(A)=0$, then $\cB(A)$ is a g-alternative algebra, the
action of $\cB(A)$ on $A$ is a derived action in $\galt$ and
according to the general result, given in \cite{CDL3},
$\cB(A)=\act(A)$. In particular, if $A$ is anticommutative (i.e.
$aa'=-a'a$, for any $a,a' \in A$) and $\ann(A)=0$, then there
exists an actor of $A$ in $\galt$. At the end of the section we
give another description of an actor of a g-alternative algebra in
terms of a certain algebra of bimultiplications, defined in
analogous way as for the case of rings or associative algebras. In
Section \ref{existence_alt} we investigate more properties of the
object $\cB(A)$ under the conditions that $A$ is anticommutative
with $\ann(A)=0$; in particular, it is proved that $\cB(A)\in
\alt$ and its action on $A$ is a derived action in $\alt$, which
give the sufficient conditions for the existence of an actor in
the category of alternative algebras. Note that in the cases of
associative algebras and rings the algebras of bimultiplications
played an important role in the extension and obstruction theories
\cite{Ho,Mc}. From this point of view the results obtained in this
paper could be applied to a cohomology and the corresponding
extension and obstruction theories of alternative algebras.

\section{Preliminary definitions and results} \label{def}
This section contains some well-known and new definitions and
results which will be used in what follows.

Let $\bC$ be a category of groups with a set of operations $\Om$ and
with a set of identities $\bE$, such that $\bE$ includes the group
identities and the following conditions hold. If $\Om_i$ is the set
of $i$-ary operations in $\Om$, then:
\begin{itemize}
\item[(a)]
$\Om=\Om_0\cup\Om_1\cup\Om_2$;
 \item[(b)] the group operations
(written additively: ($0,-,+$) are elements of $\Om_0$, $\Om_1$ and
$\Om_2$ respectively. Let $\Om_2'=\Om_2\setminus\{+\}$,
$\Om_1'=\Om_1\setminus\{-\}$ and assume that if $*\in\Om_2$, then
$\Om_2'$ contains $*^\circ$ defined by $x*^{\circ}y=y*x$. Assume
further that $\Om_0=\{0\}$;
 \item[(c)] for each $*\in\Om_2'$, $\bE$
includes the identity $x*(y+z)=x*y+x*z$;
 \item[(d)] for each
$\om\in\Om_1'$ and $*\in\Om_2'$, $\bE$ includes the identities
$\om(x+y)=\om(x)+\om(y)$ and $\om(x)*y=\om(x*y)$.
\end{itemize}
Note that the group operation is denoted additively, but it is not
commutative in general. A category $\bC$ defined above is called a
\emph{category of groups with operations}. The idea of the
definition comes from \cite{Hig} and axioms are from \cite{Orz}
and \cite{Por}. We formulate two more axioms on $\bC$ (Axiom~(7)
and Axiom~(8) in \cite{Orz}).

If $C$ is an object of $\bC$ and $x_1,x_2,x_3\in C$:

\begin{axiom}
 $x_1+(x_2*x_3)=(x_2*x_3)+x_1$, for each
$*\in\Om_2'$.
\end{axiom}

\begin{axiom} For each ordered pair
$(*,\ol{*})\in\Om_2'\times\Om_2'$ there is a word $W$ such that
\begin{gather*}
(x_1*x_2)\ol{*}x_3=W(x_1(x_2x_3),x_1(x_3x_2),(x_2x_3)x_1,\\
(x_3x_2)x_1,x_2(x_1x_3),x_2(x_3x_1),(x_1x_3)x_2,(x_3x_1)x_2),
\end{gather*}
where each juxtaposition represents an operation in $\Om_2'$.
\end{axiom}

A category of groups with operations satisfying  Axiom~1 and Axiom~2
is called a {\it category of interest} in \cite{Orz}.

Denote by $\bE_G$ the subset of identities of $\bE$ which includes
the group identities and the identities (c) and (d). We denote by
$\bC_G$ the corresponding category of groups with operations. Thus
we have $\bE_G \hookrightarrow \bE$, $\bC=(\Om,\bE)$,
$\bC_G=(\Om,\bE_G)$ and there is a full inclusion  functor
$\bC\hookrightarrow\bC_G$. The category $\bC_G$ is called a
\emph{general category of groups with operations} of a category of
interest $\bC$ in \cite{CDL1} and \cite{CDL2}.

\noindent {\bf Examples of categories of interest.} In the case of
the category of associative algebras with multiplication
represented by $*$, we have $\Om_2'=\{*,*^\circ\}$. For Lie
algebras take $\Om_2'=([\;,\;],[\;,\;]^\circ)$ (where
$[a,b]^\circ=[b,a]=-[a,b]$). For Leibniz algebras (see \cite{Lo}),
take $\Om_2'=([\;,\;],[\;,\;]^\circ)$, (here $[a,b]^\circ=[b,a]$).
It is easy to see that the categories of all these algebras are
categories of interest. In the cases of the categories of groups,
abelian groups and modules over a ring we have
$\Om_2'=\varnothing$. As it is noted in \cite{Orz} Jordan algebras
do not satisfy Axiom~2. Below we will see that the category of
alternative algebras is a category of interest as well.

\begin{defi}\cite{Orz}
Let $A$, $B\in\bC$. An \emph{extension} of $B$ by $A$ is a sequence
\begin{equation} \label{extension}
\xymatrix{0\ar[r]&A\ar[r]^-{i}&E\ar[r]^-{p}&B\ar[r]&0}
\end{equation}
in which $p$ is surjective  and  $i$ is the kernel of $p$. We say
that an extension is \emph{split}  if there is a morphism $s \colon B\lra E$ such
that $ps=1_B$.
\end{defi}

\begin{defi}\cite{Orz}
A split extension of $B$ by $A$ is called a $B$-\emph{structure} on
$A$.
\end{defi}
According to \cite{Orz}, for $A,B\in\bC$ ``a set of actions of $B$
on $A$'', means that there is a map $f_* \colon B\times A\lra A$, for
each $*\in\Om_2$. Instead of ``set of actions'' often is used for
simplicity ``an action'', and it means that there is a set of
actions $\{f_*\}_ {*\in\Om_2}$. Note that a set of actions has a
different meaning in \cite{BJK}.

A $B$-structure on $A$ induces a set of  actions of $B$ on $A$
corresponding to the operations in $\bC$. If \eqref{extension} is a split
extension, then for $b\in B$, $a\in A$ and $*\in\Om_2{}'$ we have
\begin{align}
b\cdot a &=s(b)+a-s(b),  \label{derived_dot}\\
b*a&=s(b)*a.  \label{derived_star}
\end{align}
Actions defined by \eqref{derived_dot} and \eqref{derived_star} are called \emph{derived  actions} of
$B$ on $A$ in \cite{Orz}. Under $B$-\emph{substructure} of $A$
naturally we will mean a subobject $A'$ of $A$ which is closed under
all derived actions of $B$ on $A'$, i.e. left and right derived
actions.

In the case of associative algebras over a ring $R$ a derived
action of $B$ on $A$ is a pair of $R$-bilinear maps
\begin{equation}\label{bilinear}
B\times A\lra A,\quad A\times B\lra A
\end{equation}
which we denote respectively as $(b,a)\mapsto ba$, $(a,b)\mapsto
ab$, with conditions
\begin{align*}
&\quad(b_1b_2)a=b_1(b_2a),\\
&\quad a(b_1b_2)=(ab_1)b_2,\\
&\quad(b_1a)b_2=b_1(ab_2),\\
&\quad(a_1a_2)b=a_1(a_2b),\\
&\quad b(a_1a_2)=(ba_1)a_2,\\
&\quad (a_1b)a_2=a_1(ba_2),
\end{align*}
for any $a_1, a_2\in A$ and $b_1, b_2\in B$. Note that these
identities are well-known, and they can be obtained easily from
the definition of a derived action \eqref{derived_star} and the
associativity axiom.

According to \cite{Orz}, in any category of interest, given a set
of actions of $B$ on $A$, the \emph{semidirect product} $B\ltimes
A$ is the universal algebra, whose underlying set is $B\times A$
and the operations are defined by
\begin{align*}
(b',a')+(b,a)&=(b'+b,a'+b'\cdot a),\\
(b',a')*(b,a)&=(b'*b,a'*a+a'*b+b'*a).
\end{align*}

\begin{theo}\cite{Orz} \label{derived_semi}
A set of actions of $B$ on $A$ is a set of derived actions if and
only if $B\ltimes A$ is an object of $\mathbb{C}$.
\end{theo}

\begin{defi}\cite{CDL1,CDL2}\label{actor}
For any object $A$ in $\bC$, an actor of $A$ is an object
$\act(A)$ in $\bC$, which has a derived action on $A$ in $\bC$,
and for any object $C$ of $\bC$ and a derived action of $C$ on $A$
there is a unique morphism $\vf \colon C\lra \act(A)$ with $c\cdot
a=\vf(c)\cdot a$, $c*a=\vf(c)*a$ for any $*\in\Om_2{}'$, $a\in A$
and $c\in C$.
\end{defi}

In \cite{CDL2}, for any object $A$ of a category of interest
$\bC$, we define a universal strict general actor $\cB(A)$ of $A$,
which is an object of $\bC_G$, and give the corresponding
construction. We present this construction for the case of
g-alternative algebras (see below the definition) in Section
\ref{acting}.

\begin{theo}\cite{CDL2} \label{actor_semi}
Let $\bC$ be a category of interest and $A\in\bC$. $A$ has an actor
if and only if the semidirect product $\cB(A)\ltimes A$ is an object
of $\bC$. If it is the case, then $\act(A)=\cB(A)$.
\end{theo}

From Theorems \ref{derived_semi} and \ref{actor_semi} we have
\begin{coro} \label{actor_interest}An object $A$ of a category of interest $\bC$ has an
actor if and only if $\cB(A)\in \bC$ and the action of $\cB(A)$ on
$A$ is a derived action.
\end{coro}

Recall that an \emph{alternative algebra} $A$ over a field $F$ is
an algebra which satisfies the identities $x^2y = x(xy)$ and $yx^2
= (yx)x,$ for all $x,y \in A$. These identities are known
respectively as the left and right alternative laws. We denote the
corresponding category of alternative algebras by $\alt$. Clearly
any associative algebra is alternative. The class of 8-dimensional
Cayley algebras is an important class of alternative algebras
which are not associative \cite{Sch}; commutative alternative
algebras are Jordan algebras.

Note that in all categories of algebras, considered in this paper
as categories of interest, algebras are generally without unit,
but, of course, algebras with unit are also included in these
categories. We will always note when we deal with algebras over a
field with characteristic 2, but nevertheless in such cases we
will use the sign ``$-$'' before the elements of the algebra, in
order to denote the additive inverse elements, e.g. $-a$, for the
element inverse to $a$ of the algebra $A$; this we do essentially
in order not to confuse a reader during the computations while we
apply different axioms or certain assumptions on algebras.

Here we introduce the notion of a general alternative algebra.

\begin{defi} A  \emph{general alternative algebra} (shortly
\emph{g-alternative algebra}) $A$ over a field $F$ is an algebra,
which satisfies the following two axioms for any $x, y, z \in A$
\begin{enumerate}[\hspace*{0.5cm} {A}x{i}om~${2}_1$.]
\item $ x(yz) = (xy)z + (yx)z - y(xz)$;
\item $(xy)z = x(yz) + x(zy) - (xz)y $.
\end{enumerate}
\end{defi}

These axioms are dual to each other in the sense, that if $x \circ
y=yx$, then Axiom~$2_1$ for the operation $\circ$ gives
Axiom~$2_2$ for the original operation, and obviously, Axiom~$2_2$
for the $\circ$ operation gives Axiom~$2_1$. We consider these
conditions as Axiom~2 and consequently, the category of
g-alternative algebras can be interpreted as a category of
interest, which will be denoted by $\galt$. According to the
definition of a general category of groups with operations for a
given category of interest, we obtain that $\alt_G=\galt_G$ and it
is a category of groups together with the multiplication
operation, which satisfies the identity $x(y+z)=xy+xz$.

Denote by $\overline{\galt}$ the full subcategory of
$\galt$ of those objects and homomorphisms between them,
which satisfy the following condition

E$_1$. \quad $(xy)x=x(yx)$, for any $x,y \in A$.

This identity is called the \emph{flexible law} \cite{Sch} or
\emph{flexible identity} \cite{ZSSS}.

\begin{Prop}\label{galt}
\begin{enumerate}
\item[(i)] For any field F, we have the equality $\alt=
\overline{\galt}$;

\item[(ii)] If $\ch F \neq 2$, then
$\galt=\overline{\galt}=\alt$.
\end{enumerate}
\end{Prop}

\begin{proof} (i) Let A be an alternative algebra, then we have
\[(x+y)^2z=(x+y)((x+y)z), \ \mbox{for any}\ x,y,z\in A \, ,\]
from which follows Axiom~$2_1$. Analogously, the identity
$x(y+z)^2= (x(y+z))(y+z)$ gives Axiom~$2_2$. Thus we have
$\alt \subseteq \galt$. It is well-known fact that
every alternative algebra satisfies the condition E$_1$;
see the equivalent definition of alternative algebras and the
proof of Artin's theorem, e.g. in \cite{Kur}. We include here the
proof based on the fact that alternative algebras are
g-alternative algebras. From Axiom~$2_1$, for $y=z$ we have

\begin{equation} \label{25}
xy^2=(xy)y+(yx)y-y(xy)
\end{equation}
and since $xy^2=(xy)y$, we obtain  $(yx)y=y(xy)$, which is the
condition E$_1$. Now we shall show that $\overline{\galt}\subseteq
\alt$. Let $A\in \overline{\galt}$, again from \eqref{25},
applying the condition E$_1$ we obtain $xy^2=(xy)y$. Analogously,
from Axiom~$2_2$ for $x=y$ we have $x^2z=x(xz)+x(zx)-(xz)x$; this
by E$_1$ implies $x^2z=x(xz)$, which ends the proof of ($i$).

(ii) We have only to show that if $\ch F \neq 2$, then every
g-alternative algebra satisfies the condition E$_1$. From
Axiom~$2_2$ for $y=z$, we have $(xy)y=xy^2+xy^2-(xy)y$. From this
and \eqref{25} we have $2y(xy)=2(yx)y$, applying the fact that
$\ch F\neq2$ we obtain the condition E$_1$, which ends the proof.
\end{proof}
The following example shows that in the case $\ch F=2$ we have the
strict inclusion $\alt \subset \galt$.

\begin{ex} Let $A$ be the free g-alternative algebra  generated by the one element set $\{x\}$ over a field
$F$, and $\ch F=2$. Then, according to Axiom~$2_1$, since
$(xx)x+(xx)x=0$, we only obtain that $x(xx)=-x(xx)$, as it is for
any element $a\in A$, i.e. $a=-a$, since $\ch F=2$. Analogously,
from Axiom~$2_2$, we will have $(xx)x=-(xx)x$, which shows that
$A$ is not an alternative algebra.
\end{ex}

\section{Derived actions and universal acting objects in the
categories of g-alternative and alternative algebras} \label{acting}

According to the definition of derived action in a category of
interest, in the category of g-alternative algebras over a field
$F$ we obtain the following: a derived action of $B$ on $A$ in
$\galt$ is a pair of $F$-bilinear maps \eqref{bilinear}, which we
denote by $(b,a)\mapsto b a$, $(a,b)\mapsto a b$ with the
conditions
\begin{align*}
\mbox{I}_1. \quad   b (a_1 a_2)& = (b a_1) a_2 + (a_1 b)a_2 - a_1(b a_2), \\
\mbox{I}_2. \quad (a_1 a_2) b & = a_1 (a_2 b) + a_1 (b a_2) - (a_1 b) a_2 , \\
\mbox{I}_3. \quad (b a_1) a_2 & = b (a_1 a_2) + b (a_2 a_1) - (b a_2) a_1 , \\
\mbox{I}_4. \quad a_1(a_2 b)&=(a_1 a_2)b + (a_2 a_1) b - a_2(a_1 b),\\
\mbox{II}_1. \quad (b_1 b_2) a & = b_1 (b_2 a) + b_1 (a b_2) - (b_1 a) b_2 , \\
\mbox{II}_2. \quad a (b_1 b_2) &= (a b_1) b_2 + (b_1 a) b_2 - b_1(a b_2), \\
\mbox{II}_3. \quad (a b_1) b_2& = a(b_1 b_2) + a(b_2 b_1) - (a b_2) b_1 , \\
\mbox{II}_4. \quad b_1(b_2a)&=(b_1 b_2) a + (b_2 b_1) a - b_2(b_1 a),
\end{align*}
for any $a_1, a_2\in A$ and $b_1, b_2\in B$. These identities are
obtained according to \eqref{derived_star} and Axiom~2 for
g-alternative algebras.

For the case of alternative algebras for derived actions we will
have $\mbox{I}_1-\mbox{I}_4, \mbox{II}_1-\mbox{II}_4$, and two more identities obtained from
the condition E$_1$:
\begin{align*}
& \mbox{III}_1. \quad  a(ba)=(ab)a, \\
& \mbox{III}_2. \quad  b(ab)=(ba)b,
\end{align*}
for any $a\in A, b\in B.$ In the case where $A$ is a vector space
over a field $F$ we obtain the well-known definition of an
alternative bimodule or, equivalently, of a representation of an
alternative algebra \cite{Sch}.

From Axiom~2 we easily deduce the following useful identities in
terms of Axiom~2 of the corresponding category of interest, which
express the well-known properties of alternative algebras (see
e.g. \cite{Sch}, \cite{ZSSS}). We shall use the notation like
$(1\leftrightarrow 2,-)$ in order to denote that right side of the
identity is obtained by the permutation of the first and the second
elements from the left side, and the signs of the summands are
changed; the notation like $(1\rightarrow 3\rightarrow
2\rightarrow 1)$ denotes that the right side of the identity is
obtained  by the following changes in the left side: the first
element takes the place of the third one, the third element the
place of the second one and the second element takes the place of
the first one.

\begin{align}
(1\leftrightarrow 2,-) \qquad \qquad \qquad (xy)z-x(yz)=-(yx)z+y(xz)  \label{31}\\
(1\leftrightarrow 3,-) \qquad \qquad \qquad (xy)z-x(yz)=-(zy)x+z(yx) \label{32}\\
(2\leftrightarrow
3,-) \qquad \qquad \qquad (xy)z-x(yz)=-(xz)y+x(zy) \label{33}\\
(1\rightarrow3\rightarrow2\rightarrow
1)\qquad \qquad \qquad (xy)z-x(yz)=(yz)x-y(zx) \label{34}\\
(1\rightarrow 2 \rightarrow 3\rightarrow 1)\qquad  \qquad \qquad (xy)z-x(yz)=(zx)y-z(xy) \label{35}\\
 (xy)z+z(xy)=x(yz)+(zx)y \label{36}\\
 (xy)z+(yx)z=x(yz)+y(xz) \label{37}\\
 z(xy)+z(yx)=(zx)y+(zy)x \label{38}
\end{align}
for any $x,y,z\in A$. Note that we have \eqref{31}
$\Leftrightarrow$ Axiom~$2_1$, \eqref{33} $\Leftrightarrow$
Axiom~$ 2_1$, and \eqref{34} from right to the left is the same as
\eqref{35}. In the case $\ch F\neq 2$, \eqref{32} is equivalent
to the flexible law E$_1$.

Let $\{B_i\}_{i\in I}$ be the set of all g-alternative algebras,
which have a derived action on $A$ in $\galt$. From the properties
$\mbox{I}_1-\mbox{I}_4$ it follows that identities
\eqref{31}--\eqref{38} are true for any $x,y,z\in A$, where one of
the elements from $x,y,z$ is in $B_i$ for any $i\in I$.

Let $A$ and $B$ be  g-alternative (resp. alternative) algebras and
$B$ has a derived action on $A$ in $\galt$ (resp. $\alt$).
According to the general definition of a semidirect product in a
category of interest given in Section \ref{def}, $B\ltimes A$ is a
g-alternative (resp. an alternative) algebra, whose underlying set
is $B\times A$ and whose operations are given by
\begin{align*}
(b',a')+(b,a)&=(b'+b,a'+ a),\\
(b',a')(b,a)&=(b'b,a'a+a'b+b'a).
\end{align*}
According to Definition \ref{actor}, for any g-alternative (resp.
alternative) algebra $A$, an \emph{actor} of $A$ is an object
$\act(A)\in \galt$ (resp. $\act(A)\in \alt$),
which has a derived action on $A$ in $\galt$ (resp.
$\alt$) and for any g-alternative (resp. alternative)
algebra $C$ and a derived action of $C$ on $A$, there is a unique
homomorphism $\vf \colon C\lra \act(A)$ with $c a=\vf(c) a$, for
any $a\in A$ and $c\in C$.

Here we give a construction of a universal strict general actor
$\cB(A)$ of a g-alternative algebra $A$, which is a special case
of the construction given in \cite{CDL2} for categories of
interest. In this case we have only two binary operations: the
addition, denoted by `` + '', and the multiplication, which will
be denoted here by dot ``$\cdot$''. Note that this sign was
omitted usually in above expressions for g-alternative algebras,
and it will be so in the next sections as well, when there is no
confusion. Since the addition is commutative, the action
corresponding to this operation is trivial. Thus we will deal only
with actions, which are defined by multiplication according to
\eqref{derived_star}, and this action will be denoted by $\cdot$
as well
\begin{align*}
b\cdot a&=s(b)\cdot a.
\end{align*}
Consider all split extensions of $A$ in $\galt$
\[  \xymatrix{E_j \colon 0\ar[r]&A\ar[r]^-{i_j}&C_j\ar[r]^-{p_j}&B_j\ar[r]&0},\quad j\in\bJ \, .  \]
Note that it may happen that $B_j=B_k=B$, for $j\neq k$, in this
case the corresponding extensions derive different actions of $B$ on
$A$.
 Let $\{(b_j\cdot, \cdot b_j)|b_j\in
B_j\}$ be the set of all pairs of $F$-linear maps $A\rightarrow
A$, defined by the action of $B_j$ on $A$. For any element $b_j\in
B_j$ denote ${\bbb}_j=(b_j\cdot,\cdot b_j)$. Let
$\bB=\{{\bbb}_j|b_j\in B_j,\;j\in\bJ\}$.

According to Axiom~2 from the definition of $\galt$ as a category
of interest, we define the multiplication,
${\bbb}_i\cdot{\bbb}_k$, for the elements of $\bB$ by the
equalities
\begin{align*}
&({\bbb}_i\cdot{\bbb}_k)\cdot(a)=b_i\cdot(b_k\cdot
a)+b_i\cdot(a\cdot b_k)-(b_i\cdot a)\cdot b_k \, ,\\
&(a)\cdot({\bbb}_i\cdot{\bbb}_k)=(a\cdot b_i)\cdot b_k+(b_i\cdot
a)\cdot b_k-b_i\cdot(a\cdot b_k) \, .
\end{align*}
For $b=\bbb_{i_1}\cdot \ldots \cdot\bbb_{i_n}$ with certain brackets,
$b\cdot a$ is defined in a natural way step by step according to
brackets and to above defined equalities.

 We define the operation of addition by
\begin{align*}
&({\bbb}_i+{\bbb}_k)\cdot(a)=b_i\cdot a+b_k\cdot a.
\end{align*}
For the unary operation ``-'' we define
\begin{align*}
&(-{\bbb}_k)\cdot(a)=-(b_k\cdot a),\\
&(-b)\cdot(a)=-(b\cdot(a)),\\
&- (b_1+\dots +b_n)=-b_n-\dots-b_1,
\end{align*}
where $b,b_1,\dots,b_n$ are certain combinations of the dot
operation on the elements of $\bB$, i.e. the elements of the type
$\bbb_{i_1}\cdot \ldots \cdot\bbb_{i_n}$, where we mean certain
brackets, $n>1$. Obviously, the addition is commutative.

Denote by $\cB'(A)$ the set of all pairs of $F$-linear maps
obtained by performing all kinds of above defined operations on
the elements of $\bB$. We define the following relation: we will
write $b\sim b'$, for $b,b'\in \cB'(A),$ if and only if
$b\cdot(a)=b'\cdot(a)$, for any $a\in A$. This relation is a
congruence relation on $\cB'(A)$, i.e. it is compatible with the
operations we have defined in $\cB'(A)$. We define
$\cB(A)=\cB'(A)/\sim$. The operations defined on $\cB'(A)$ induce
the corresponding operations on $\cB(A)$. For simplicity we will
denote the elements of $\cB(A)$ by the same letters $b, b'$, etc.
instead of the classes $clb, clb'$, etc.

According to the general result \cite[Proposition 4.1]{CDL2}
$\cB(A)$ is an object of $\galt_G$; moreover, it is obvious, that
$\cB(A)$ is an $F$-algebra.

 Define the set of actions of $\cB(A)$ on $A$ in a natural way:
for $b\in \cB(A)$ we define $b\cdot a=b\cdot(a)$. Thus if
$b=\bbb_{i_1}\cdot_1 \ldots \cdot_{n-1}\bbb_{i_n}$, where we mean
certain brackets, we have
\begin{align*}
&b\cdot\,a=(\bbb_{i_1}\cdot_1 \ldots \cdot_{n-1}\bbb_{i_n})\,\cdot\,(a),\\
\end{align*}
where the right side of the equality is defined according to the
brackets and Axiom~2. For $b_k\in B_k$, $k\in\bJ$, we have
\begin{align*}
\bbb_k\cdot a=\bbb_k\cdot(a)=b_k\cdot a.
\end{align*}
Also
\[(b_1+b_2+ \dots +b_n)\cdot
a=b_1\cdot(a)+ \dots +b_n\cdot(a), \; \mbox{for} \; b_i\in \cB(A),\;
i=1, \ldots, n \, .\]

It is easy to check (see  \cite[Proposition 4.2]{CDL2}  for the
general case) that the set of actions of $\cB(A)$ on $A$ is a set
of derived actions in $\galt_G$.

The construction of a universal strict general actor of an
alternative algebra $A$ in $\alt$ is analogous to the one of
$\cB(A)$; in this case we consider all split extensions $E_j$ of
$A$ in $\alt$, i.e we will deal with the family $\{B_i\}_{i\in I}$
of all alternative algebras, which have a derived action on $A$ in
$\alt$. The corresponding object will be denoted by
$\cB_{\alt}(A)$, and it is an object of $\alt_G$.

For any $A\in \galt$, define the map $d \colon A\lra \cB(A)$ by
$d(a)=\aaa$, where $\aaa=\{(a \cdot ,\cdot a)\}$. Thus we have by
definition
\begin{equation*}
d(a)\cdot a'=a\cdot a', \quad  a'\cdot d(a)=a'\cdot a, \ \mbox{for any} \
a,a'\in A.
\end{equation*}\
According to the general results (see the case of a category of
interest, \cite[Lemma 4.5 and Proposition 4.6]{CDL2}) $d$ is a
homomorphism in $\galt_G$ and moreover $d \colon A\lra \cB(A)$ is a
crossed module in $\galt_G$ with certain universal property.

From the general definition of a center in categories of interest
\cite{CDL2}, (cf. \cite{Orz}), for any g-alternative algebra $A$,
we obtain that the \emph{center} of $A$ is defined by
\[Z(A)=\Ker d=\{z\in A \mid \mbox{for all}\ a\in A,
 az=za=0\} \, .\]
Therefore the center in this case is a left-right annulator of $A$,
 denoted by $\ann(A)$ (see \cite{Mc} for the case of rings).

\section{$\soci(A)$, $\asoci(A)$ and their properties}\label{properties}

 For any $x, y, z \in A\bigcup\cB(A)$, $A\in \alt$, consider
 the following types of elements:

 \begin{itemize}

 \item[(sac)]  $x(yz)+x(zy), \, (yz)x+(zy)x$ ;

  \item[(as)]  $x(yz)-(xy)z$;

  \item[(aas)] $x(yz)+(xy)z$;

  \item[(ap)]  $x(yz)+y(xz)$, \, $(yz)x+(yx)z$.
 \end{itemize}
 Let ${\{B_i\}}_{i\in I}$ be the family of g-alternative algebras which have a
 derived action on $A$. According to the above notation we define
 the following sets in $A$.

  \begin{itemize}
  \item $S_1^{\rm sac}$  (resp.  $S_2^{\rm sac}$)\,: the set of elements of $A$
 of the type (sac), where one element (resp. two elements) from the triple
 $(x,y,z)$ is (resp. are) in $\bigcup_{i\in I}{B_i}$.

\item $\bar{S}_1^{\rm sac}$  (resp.  $\bar{S}_2^{\rm sac}$)\,: the set of elements of $A$
 of the type (sac), where one element (resp. two elements)  from the triple
 $(x,y,z)$ is (resp. are) in $\cB(A)$.

 \item $S_1^{\rm as}$  (resp.  $S_2^{\rm as}$)\,: the set of elements of $A$
 of the type (as), where one element (resp. two elements) from the triple
 $(x,y,z)$ is (resp. are) in $\bigcup_{i\in I}{B_i}$.

\item $\bar{S}_1^{\rm as}$  (resp.  $\bar{S}_2^{\rm as}$)\,: the set of elements of $A$
 of the type (as), where one element (resp. two elements)  from the triple
 $(x,y,z)$ is (resp. are) in $\cB(A)$.

 \item $S_1^{\rm aas}$  (resp.  $S_2^{\rm aas}$)\,: the set of elements of $A$
 of the type (aas), where one element (resp. two elements) from the triple
 $(x,y,z)$ is (resp. are) in $\bigcup_{i\in I}{B_i}$.

 \item $\bar{S}_1^{\rm aas}$  (resp.  $\bar{S}_2^{\rm aas}$)\,: the set of elements of $A$
 of the type (aas), where one element (resp. two elements)  from the triple
 $(x,y,z)$ is (resp. are) in $\cB(A)$.

\item $S_1^{\rm ap}$  (resp.  $S_2^{\rm ap}$)\,: the set of elements of $A$
 of the type (ap), where one element (resp. two elements) from the triple
 $(x,y,z)$ is (resp. are) in $\bigcup_{i\in I}{B_i}$.

 \item $\bar{S}_1^{\rm ap}$  (resp.  $\bar{S}_2^{\rm ap}$)\,: the set of elements of $A$
 of the type (ap), where one element (resp. two elements)  from the triple
 $(x,y,z)$ is (resp. are) in $\cB(A)$.
  \end{itemize}

 \begin{defi} For any $A\in \galt$ define $\soci(A)$ as
 the $\cB(A)$-substructure of $A$ generated by the set $S_1^{\rm sac}$.
\end{defi}
 \begin{defi}
 For any $n\geq1$ denote $\asoci^n(A)=\{x\in
 A|(\dots((xa_1)a_2)\dots a_n)\in \soci(A),\, \text{for any} \,\, a_1,a_2,\dots,a_n\in A\}$. \
Define $\asoci(A) = \bigcup _{n\geq 1}\asoci^n(A)$.
 \end{defi}

 From the definition it follows that $\soci(A)\subseteq
 \asoci^1(A)$ and $\asoci^n(A)\subseteq
 \asoci^{n+1}(A)$  for any $n$.

 \begin{Lem}\label{substructure}
 Let $A'$ be a $B_i$-substructure of $A$ for any $i\in I$. Then $A'$
 is a $\cB(A)$-substructure of $A$ in $\galt_G$.
 \end{Lem}
 \begin{proof} Let $b=b_ib_j$ be an element of $\cB(A)$. For any $x\in A'$ we have
\[(b_ib_j)x=b_i(b_jx)+b_i(xb_j)-(b_ix)b_j \, .\]
By the condition of the lemma, every element on the right side is an
 element of $A'$, therefore $(b_ib_j)x\in A'$. By induction, in
 analogous way, it can be proved that $bx\in A'$ for any $b\in \cB(A)$.
 \end{proof}

 \begin{Lem}\label{asoci}
 $\asoci(A)$ is a two sided ideal of $A$.
\end{Lem}
\begin{proof} It is obvious that if $x\in \asoci(A)$, then
$xa\in \asoci(A)$ for any $a\in A$. Now we shall show that $ax\in
\asoci(A)$. Since $x\in \asoci(A)$, there exists a natural number
$k$, such that  for any $a_1,a_2,\dots,a_k \in A$, $(\dots
((-xa_1)a_2)\dots )a_k\in \soci(A)$; in particular, $(\dots
((-xa)a_1)\dots )a_{k-1}\in \soci(A)$. We have $(ax)a_1\simeq
(-xa)a_1$. By the definition $\soci(A)$ is an ideal of $A$,
therefore $(\dots ((-xa)a_1)\dots)a_{k-1}\simeq
(\dots((ax)a_1)\dots)a_{k-1}$. From this it follows that
$(\dots((ax)a_1)\dots)a_{k-1}\in \soci(A)$, which ends the proof.
\end{proof}

 \begin{Prop}\label{asoci_sub} For any $A\in \galt$, $\asoci(A)$ is a
 $\cB(A)$-substructure of $A$.
\end{Prop}
\begin{proof} By Lemma \ref{substructure} we have only to prove that for any $b_i\in
B_i$, $i\in I$, if $x\in \asoci(A)$, then $xb_i$, $b_ix\in
\asoci(A)$. First consider the case, where $x\in
\asoci^1(A)$. For any $a\in A$ we have
\[(xb_i)a=x(b_ia)+x(ab_i)-(xa)b_i \, .\]
The sum of the first two summands on the right side is in
$\soci(A)$; since $xa_1\in \soci(A)$ and $\soci(A)$, by
definition, is a $\cB(A)$-substructure in $A$, the third summand
$(xa)b_i$ is in $\soci(A)$ as well. From this we conclude that
$(xb_i)a\in \soci(A)$, which means that $(xb_i)\in \asoci(A)$.

Consider the case where $x\in \asoci^2(A)$. We have
$(xa_1)a_2\in \soci(A)$, for any $a_1, a_2\in A$. We compute
\[((xb_i)a_1)a_2=(x(b_ia_1)+x(a_1b_i))a_2-((xa_1)b_i)a_2\, .\]
By Lemma \ref{asoci}, $xa_1\in \asoci^1(A)$, thus from the above
proved case we obtain $(xa_1)b_i\in \soci(A)$, and therefore
$((xa_1)b_i)a_2$ is in $\soci(A)$ as well. Moreover
$(x(b_ia_1)+x(a_1b_i))a_2\in \soci(A)$ since
$\soci(A)$ is an ideal of $A$. From this we conclude that
$((xb_i)a_1)a_2\in \soci(A)$. In this way by induction we
show that for any natural number $n$ and $x\in \asoci^n(A)$,
$xb_i\in \asoci(A)$. In analogous way is proved that
$b_ix\in \asoci(A)$, which ends the proof.
\end{proof}
\noindent
 \textbf{Notation.} For any set $X$ of elements of $A$, $A\in \galt$, we will
 write $X\simeq 0$ if and only if $X\subseteq \asoci(A)$.
 We will write $x_1\simeq x_2$ (resp.$x_1\sim x_2$), for $x_1, x_2 \in A$, if and only if
 $x_1-x_2 \in \asoci(A)$ (resp. $x_1-x_2 \in \soci(A)$).
 Therefore $x_1\sim x_2$ implies $x_1\simeq x_2$.

 By Proposition \ref{asoci_sub} and the definition of $\soci(A)$, \,$\simeq$ \, and \,$\sim$ \, are  congruence relations for the elements
 in $A$.
 \begin{Lem}\label{aas1}
\begin{enumerate}
\item[(i)] $S_1^{aas}\sim0$ and  $S_1^{as}\simeq0$.

\item[(ii)] If $A$ is an anticommutative g-alternative algebra, then $A$
is antiassociative and second level associative, i.e.
$((xy)z)a=(x(yz))a$, for any $a,x,y,z\in A$.

\item[(iii)] If $A$ is an anticommutative g-alternative algebra
and $\ann(A)=0$, then $A$ is antiassociative, associative and
$2x(yz)=2(xy)z=0$, for any $x,y,z\in A$.

\item[(iv)] If $A$ is anticommutative g-alternative algebra over a
field $F$ with $\ch F \neq 2$ and $\ann(A)=0$, then
$A=0$.
\end{enumerate}
\end{Lem}
\begin{proof} (i) First we show that $S_1^{aas}\sim0$. We have
$(xy)z+x(yz)\sim -(xz)y-y(xz)\sim -(xz)y+(xz)y=0$.

For $S_1^{as}\simeq0$, from Axiom $2_1$ we obtain $(x(yz))a\sim
-(xa)(yz)$.

Applying  antiassociativity up to congruence relation (i.e. the
fact that $S_1^{aas}\sim0$) and the definition of $\soci(A)$ we
obtain $((xy)z)a\sim -((xy)a)z\sim ((xa)y)z\sim -(xa)(yz)$.
Therefore we have $(x(yz))a-((xy)z)a\sim 0,$ which proves that
$S_1^{as}\simeq0$.

(ii) Antiassociativity of $A$ is a special case of
$S_1^{aas}\sim0$ in (i), where ``$\sim$'' is replaced by ``='',
since $A$ is anticommutative.  The proof of the second level
associativity of $A$ is a special case of the proof of
$S_1^{as}\simeq0$ in (i), where $x, y,z\in A$ and ``$\sim$'' is
replaced by ``=''again.

(iii) Since $\ann(A)=0$, second level associativity of $A$ implies
associativity and therefore, applying (ii) we have
$x(yz)=(xy)z=-x(yz)$, from which follows the result.

(iv) Since from (iii) $(2\cdot 1_F)x(yz)=0$ and $\ch F
\neq 2$, it follows that $x(yz)=0$ for any $x,y,z\in A$, which
implies that $z=0$, for any $z\in A$, since $\ann(A)=0$.
\end{proof}

\begin{Prop}\label{bar_sac2}
$\bar{S}_2^{\rm sac}\simeq 0$.
\end{Prop}
 The proof is based on several lemmas.

 \begin{Lem}\label{sac2}
 $S_2^{\rm sac}\simeq 0$.
\end{Lem}
\begin{proof}  We shall prove the congruence relation for the
elements of the type $x(yz)+x(zy)$, and the case $(yz)x+(zy)x$ is
left to the reader.

Consider the case where $x=b_i, y=b_j$. For any $a\in A$ we have
\[(b_i(b_jz)+b_i(zb_j))a\sim -(b_ia)(b_jz)-(b_ia)(zb_j) \, .\]
The right side is an element of $S_1^{\rm sac}$ and therefore it
is an element of $\soci(A)$, which proves that
$(b_i(b_jz)+b_i(zb_j))\sim 0$. Analogously, it can be proved that
$((xb_i)b_j+(b_ix)b_j)a\sim0$ and therefore
$(xb_i)b_j+(b_ix)b_j\simeq0$.

Consider the case $y=b_i, z=b_j$. Thus we have to show that
$x(b_ib_j)+x(b_jb_i)\in \asoci(A)$. For any $a\in A$ we have
$(x(b_ib_j)+x(b_jb_i))a=((xb_i)b_j)a+((b_ix)b_j)a-(b_i(xb_j))a+((xb_j)b_i)a+
((b_jx)b_i)a-(b_j(xb_i))a$. Applying the case noted above (i.e.,
$((xb_i)b_j+(b_ix)b_j)a\sim0$) we obtain that the right side is
$\sim$-congruent to the following one $-(b_i(xb_j))a-(b_j(xb_i))a=
-b_i((xb_j)a)-b_i(a(xb_j))+(b_ia)(xb_j)-b_j((xb_i)a)-b_j(a(xb_i))+(b_ja)(xb_i)\sim
(b_ia)(xb_j)+(b_ja)(xb_i)$.

For any $a'\in A$ we have
$((b_ia)(xb_j))a'\sim-((b_ia)a')(xb_j)\sim(((b_ia')a)xb_j)\sim
-(b_ia')(a(xb_j))\sim(b_ia')(x(ab_j))\sim-(b_ia')(x(b_ja))\sim
(b_ia')((b_ja)x)\sim-((b_ia')(b_ja))x\sim((b_ja)(b_ia'))x\sim-(b_ja)((b_ia')x)\sim
(b_ja)((b_ix)a')\simeq((b_ja)(b_ix))a'\sim((b_ja)(xb_i))a'$.

Here we applied the fact that $\soci(A)\sim0$ and Lemma \ref{aas1} (i).
Thus we obtain that $((b_ia)(xb_j)+(b_ja)(xb_i))a'\simeq0$, which
gives that $((x(b_ib_j)+x(b_jb_i))a)a'\in \asoci(A)$, for any $a,
a'\in A$ and therefore $x(b_ib_j)+x(b_jb_i)\in \asoci(A)$.
\end{proof}

 \begin{Lem}\label{bar sac1}
 $\bar{S}_1^{\rm sac}\simeq 0$.
\end{Lem}
\begin{proof} We shall show that $(b_ib_j)(yz)+(b_ib_j)(zy)\in
\asoci(A)$ and other cases are left to the reader.
$(b_ib_j)(yz)+(b_ib_j)(zy)=b_i(b_j(yz))+b_i((yz)b_j)-(b_i(yz))b_j+b_i(b_j(zy))+
b_i((zy)b_j)-(b_i(zy))b_j\simeq-(b_i(yz))b_j-(b_i(zy))b_j\simeq(b_i(zy))b_j-(b_i(zy))b_j\simeq0$.

Here we applied Proposition \ref{asoci_sub} and Lemma \ref{sac2}.
\end{proof}
 \begin{Lem}\label{aas2}
 $S_2^{\rm as}\simeq0$,  \quad $S_2^{\rm aas}\simeq 0$.
\end{Lem}
\begin{proof} We shall prove that $b_k(yb_l)-(b_ky)b_l\simeq0$, for
any $y\in A$, $b_k\in B_k$ and $b_l\in B_l$ $k, l\in I$; other
cases are proved in analogous ways applying Lemma \ref{bar sac1}. We have

$(b_k(yb_l))a-((b_ky)b_l)a\simeq-(b_ka)(yb_l)+((b_ky)a)b_l\simeq
((b_ka)y)b_l+((b_ky)a)b_l\simeq-((b_ky)a)b_l+((b_ky)a)b_l\simeq0$.
\end{proof}

 Applying these lemmas we prove Proposition \ref{bar_sac2}.

\begin{proof}(Proposition \ref{bar_sac2}) We shall show that
$(b_ib_j)(y(b_kb_l))+(b_ib_j)((b_kb_l)y)\in \asoci(A)$.

We have
$(b_ib_j)(y(b_kb_l))+(b_ib_j)((b_kb_l)y)=(b_ib_j)((yb_k)b_l+(b_ky)b_l-b_k(yb_l))+
(b_ib_j)(b_k(b_ly)+b_k(yb_l)-(b_ky)b_l)$.

Applying Proposition \ref{asoci_sub} and Lemma \ref{sac2} we
obtain that this expression is $\simeq$-congruent to the following
$-(b_ib_j)(b_k(yb_l))-(b_ib_j)((b_ky)b_l)$, which by Lemma
\ref{aas2} is $\simeq$-congruent to $0$.
\end{proof}

 \begin{Lem}\label{bar_aas1}
 $\bar{S}_1^{\rm as}\simeq 0$, \quad $\bar{S}_1^{\rm aas}\simeq 0$.
\end{Lem}
\begin{proof} We shall prove that $((b_ib_j)y)z-(b_ib_j)(yz)\simeq0$, $y, z\in A$.
We have

$((b_ib_j)y)z-(b_ib_j)(yz)=(b_i(b_jy)+b_i(yb_j)-(b_iy)b_j)z-b_i(b_j(yz))-
b_i((yz)b_j)+(b_i(yz))b_j\simeq -((b_iy)b_j)z+(b_i(yz))b_j\simeq
-((b_iy)b_j)z-(y(b_iz))b_j\simeq((b_iy)z)b_j-(y(b_iz))b_j\simeq
-((yb_i)z)b_j-(y(b_iz))b_j\simeq-(y(b_iz))b_j-(y(b_iz))b_j=0$.

Other cases are proved in similar ways applying Proposition \ref{bar_sac2}.
\end{proof}

\begin{Prop}\label{bar_as2}
$\bar{S}_2^{\rm as}\simeq 0$.
\end{Prop}
{\begin{proof}  We shall show that
$(b_ib_j)(y(b_kb_l))-((b_ib_j)y)(b_kb_l)\simeq0$. The general case
can be proved by application Lemmas \ref{aas2} and \ref{bar_aas1}. We apply Lemma
\ref{aas2} and obtain

$(b_ib_j)(y(b_kb_l))-((b_ib_j)y)(b_kb_l)\simeq(b_ib_j)((yb_k)b_l)-(((b_ib_j)y)b_k)b_l\simeq
b_i(b_j((yb_k)b_l))-((b_i(b_jy))b_k)b_l\simeq
b_i((b_j(yb_k))b_l)-((b_i(b_jy))b_k)b_l\simeq
(b_i(b_j(yb_k)))b_l-((b_i(b_jy))b_k)b_l\simeq(b_i(b_j(yb_k)))b_l-((b_i(b_jy))b_k)b_l\simeq
((b_i(b_jy))b_k)b_l-((b_i(b_jy))b_k)b_l\simeq0$.
\end{proof}

In analogous ways are proved the following statements.
\begin{Prop}\label{bar_aas2}
$\bar{S}_2^{ \rm aas}\simeq 0$.
\end{Prop}

\begin{Lem}\label{bar_ap1}
$\bar{S}_1^{\rm ap}\simeq 0$.
\end{Lem}

\begin{Prop}\label{bar_ap2}
$\bar{S}_2^{\rm ap}\simeq 0$.
\end{Prop}

 \section{Sufficient conditions for $\cB(A)\in \galt$} \label{sufficient}

 We know from Section \ref{def} that $\cB(A)$ is an object in $\galt_G
 $. We begin
 this section with an example, which shows that the elements of $\cB(A)$ generally
 do not satisfy Axiom~$2_1$ and Axiom~$2_2$, and therefore  $\cB(A)$ is not a
 g-alternative algebra in general.

 \begin{ex}\label{ex_prod}  Let $A$ and  $\Lambda$ be associative algebras over a field $F$ with $\ch F \neq 2, 3$; let $A$
 be anticommutative, $\ann(A)\neq 0$, and $A$
 has a
 derived action of $\Lambda$ in the category of associative
 algebras, such that $\lambda aa'=0, \lambda\lambda ' a=0$ and
 $\lambda a=-a\lambda$, for any $\lambda, \lambda'\in \Lambda,$ and $a,a'\in
A$. For example one can take $\Lambda=A$, since
 anticommutativity of $A$ implies antiassociativity, and together
 with $\ch F \neq 2$ it gives that $aa'a''=0$ for any $a,a',
 a''\in A$. Let $R$ be a g-alternative algebra with unit 1, which acts on $A$
 in $\galt$, in such a way that $1a=a1=a$ for any $a\in A$.

Let $A\times A$ be the product associative algebra. Consider the
 following actions of $R$ and $\Lambda$ on $A\times A$:
\[r(a,a')=(ra,0), \quad
(a,a')r=(ar,0), \qquad
\lambda(a,a')=(0,\lambda a), \quad
(a,a')\lambda=(0,a\lambda),\]
for any $r\in R, \lambda\in \Lambda$ and $(a,a')\in A\times A$.
 It is obvious that the action of $R$ on $A\times A$ is a derived action
and, it can be easily checked, that the action of $\Lambda$ on
$A\times
 A$ is a derived action as well.
 We will show that the equality
\[b_i(b_jb_k)=(b_ib_j)b_k+(b_jb_i)b_k-b_j(b_ib_k) \, ,\]
where $b_i\in B_i$, $b_j\in B_j$, $b_k\in B_k$ doesn't hold in
 general.

 Consider the case, where $b_i=\lambda$, $b_j=b_k=1$,
 $\lambda\in \Lambda$,  and $1$ is unit of $R$.

 We first compute the results of the following actions and obtain:
\begin{align*}
(1 \lambda)(a,a') &=-(0,a\lambda), &
(\lambda 1)(a,a') &=(0, 2\lambda a), \\
(a,a')(1\lambda) &=(0,2a\lambda), &
(a,a')(\lambda 1) &=-(0, \lambda a), \\
\lambda (1(a,a')) &=(0,\lambda a), &
\lambda((a,a')1) &=(0,\lambda a), \\
(1(a,a'))\lambda &=(0,a\lambda), &
((a,a')1)\lambda &=(0,a\lambda), \\
r((a,a')\lambda) &=(0,0), &
r(\lambda(a,a')) &=(0,0), \\
((a,a')\lambda)r &=(0,0), &
(\lambda(a,a'))r &=(0,0),
 \end{align*}

 for any $(a,a`)\in A\times A, r\in R$ and unit $1$ of $R$.

We shall show that the following equality is not true in general
\[(\lambda(1 \cdot 1))(a,a')= ((\lambda
1)1)(a,a')+((1\lambda)1)(a,a')-(1(\lambda 1))(a,a') \, .\]
The computations of both sides give that this equality is
equivalent to the following one \quad
$2\lambda a=4\lambda a -2 a\lambda -\lambda a=0$.

Since $\lambda a=-a\lambda$ and $\ch F\neq 3$, this equality gives
$\lambda a=0$, which is not true in general.

This shows that in the case of this example Axiom~$2_1$ is not
true. The same example Axiom~$2_2$ is not true as well.
 This can be checked by analogous computations or
we can apply the duality in the following way. Define in $A$ the
dual operation by $x \circ y=yx$. Axiom~$2_2$ for the original dot
operation is equivalent to the Axiom~$2_1$ for the dual
``$\circ$'' operation. But since both operations have the  same
properties, we can conclude from the above prove that Axiom~$2_1$
is not true for the ``$\circ$'' operation.
\end{ex}

 We are looking for the sufficient conditions for $\cB(A)$ to be a
 g-alternative algebra, i.e. for the conditions under which the
 elements of $\cB(A)$ satisfy  Axioms~$2_1$ and $2_2$. For any $b_1,
 b_2, b_3 \in \cB(A)$ we must have the following identities
 \begin{enumerate}
 \item[] B1. \quad  $-(b_1(b_2b_3))a+((b_1b_2)b_3)a+((b_2b_1)b_3)a-(b_2(b_1b_3))a=0$ ,
\item[] B2.  \quad $-a(b_1(b_2b_3))+a((b_1b_2)b_3)+a((b_2b_1)b_3)-a(b_2(b_1b_3)=0$;
 \end{enumerate}
 and the dual identities

  \begin{enumerate}
  \item[] B3= B2$^\circ$.  \quad  $-((b_1b_2)b_3)a+(b_1(b_2b_3))a+(b_1(b_3b_2))a-((b_1b_3)b_2)a=0$,

  \item[] B4= B1$^\circ$.  \quad  $-a((b_1b_2)b_3)+a(b_1(b_2b_3)+a(b_1(b_3b_2))-a((b_1b_3)b_2)=0$.
 \end{enumerate}
 First we compute the left side of identity B1. By the definition of the multiplication
in $\cB(A)$ we obtain:
 \begin{multline} \label{B1}
-(b_1(b_2b_3))a+((b_1b_2)b_3)a+((b_2b_1)b_3)a-(b_2(b_1b_3))a\\
= -b_1((b_2b_3)a)-b_1(a(b_2b_3))
+(b_1a)(b_2b_3) + (b_1b_2)(b_3a)
 +(b_1b_2)(ab_3)
  -((b_1b_2)a)b_3\\
  + (b_2b_1)(b_3a)+(b_2b_1)(ab_3)-((b_2b_1)a)b_3-
 b_2((b_1b_3)a)-b_2(a(b_1b_3))+(b_2a)(b_1b_3) \\
  = -b_1(b_2(b_3a))-b_1(b_2(ab_3))+b_1((b_2a)b_3)-
 b_1((ab_2)b_3)-b_1((b_2a)b_3)+b_1(b_2(ab_3)) \\
 +((b_1a)b_2)b_3+(b_2(b_1a))b_3-b_2((b_1a)b_3)+
 b_1(b_2(b_3a))+b_1((b_3a)b_2)-(b_1(b_3a))b_2 \\
 +  b_1(b_2(ab_3))+b_1((ab_3)b_2)-(b_1(ab_3))b_2-
 (b_1(b_2a))b_3-(b_1(ab_2))b_3+((b_1a)b_2)b_3 \\
 + b_2(b_1(b_3a))+b_2((b_3a)b_1)-(b_2(b_3a))b_1+
 b_2(b_1(ab_3))+b_2((ab_3)b_1)-(b_2(ab_3))b_1 \\
 - (b_2(b_1a))b_3-(b_2(ab_1))b_3+((b_2a)b_1)b_3-
 b_2(b_1(b_3a))-b_2(b_1(ab_3))+b_2((b_1a)b_3) \\
 - b_2((ab_1)b_3)-b_2((b_1a)b_3)+b_2(b_1(ab_3))+
 ((b_2a)b_1)b_3+(b_1(b_2a))b_3-b_1((b_2a)b_3).
  \end{multline}
The left side of identity B2 gives the following
\begin{multline} \label{B2}
 -a(b_1(b_2b_3))+a((b_1b_2)b_3)+a((b_2b_1)b_3)-a(b_2(b_1b_3))\\
 = -(ab_1)(b_2b_3)-(b_1a)(b_2b_3)+b_1((ab_2)b_3)-
((b_1a)b_2)b_3-(b_1(b_2a))b_3+b_1((b_2a)b_3)\\
+ b_2((ab_1)b_3)+b_2((b_1a)b_3)-b_2(b_1(ab_3))-
(ab_2)(b_1b_3)-(b_2a)(b_1b_3)+b_2(a(b_1b_3))\\
= -((ab_1)b_2)b_3-(b_2(ab_1))b_3+b_2((ab_1)b_3)-((b_1a)b_2)b_3-(b_2(b_1a))b_3+b_2((b_1a)b_3)\\
+ b_1((ab_2)b_3)+b_1((b_2a)b_3)-b_1(b_2(ab_3))+((ab_1)b_2)b_3+((b_1a)b_2)b_3-(b_1(ab_2))b_3\\
+(b_1(b_2a))b_3+(b_1(ab_2))b_3-((b_1a)b_2)b_3-b_1(b_2(ab_3))-b_1((ab_3)b_2)+(b_1(ab_3))b_2\\
+((ab_2)b_1)b_3+((b_2a)b_1)b_3-(b_2(ab_1))b_3+(b_2(b_1a))b_3+(b_2(ab_1))b_3-((b_2a)b_1)b_3\\
- b_2(b_1(ab_3))-b_2((ab_3)b_1)+(b_2(ab_3))b_1-((ab_2)b_1)b_3-(b_1(ab_2))b_3+b_1((ab_2)b_3)\\
- ((b_2a)b_1)b_3-(b_1(b_2a))b_3+b_1((b_2a)b_3)+b_2((ab_1)b_3)+b_2((b_1a)b_3)-b_2(b_1(ab_3)).
  \end{multline}

The identities B3 and B4 give the duals to the expressions
\eqref{B2} and \eqref{B1} respectively.

It is easy to see that all obtained expressions are the combinations
of the elements of the following type
\begin{enumerate}
\item[] $\mathbf{A}1=b_1(b_2(b_3a))+b_1(b_2(ab_3))$,

\item[] $\mathbf{A}2=b_1((ab_2)b_3)+b_1((b_2a)b_3)$,

 \item[] $\mathbf{A}3=((b_1a)b_2)b_3+(b_2(b_1a))b_3$,

\item[]  $\mathbf{A}4=b_1(b_2(ab_3))+b_2(b_1(ab_3))$,

\item[]  $\mathbf{A}5=((b_1a)b_2)b_3+((b_2a)b_1)b_3$,

\item[]  $\mathbf{A}6=b_1(b_2(b_3a))+b_1((b_3a)b_2)$,

\item[]  $\mathbf{A}7=b_1(b_2(ab_3))+b_1((ab_3)b_2)$,

\item[]  $\mathbf{A}8=(b_1(b_2a))b_3+(b_1(ab_2))b_3$,

\item[]  $\mathbf{A}9=((ab_1)b_2)b_3+(b_2(ab_1))b_3$,

\item[]  $\mathbf{A}10=((ab_3)b_1)b_2+((ab_3)b_2)b_1$,

 \item[] $\mathbf{A}11=b_3(b_2(ab_1))+b_3(b_1(ab_2))$,
\end{enumerate}
 where $b_1, b_2, b_3 \in \cB(A),  a\in A$.

\begin {theo} If for any $b_1, b_2, b_3 \in \cB(A), a\in A$, we have
the equalities $\mathbf{A}i=0$ for $i=1,\dots,11$, then $\cB(A)$ is
a g-alternative algebra.
\end{theo}

\begin{proof} The proof follows directly from the identities \eqref{B1},
\eqref{B2} and the dual identities, which proof that under the
conditions of the theorem we have Axioms~$2_1$ and $2_2$ for the
elements of $\cB(A)$.
\end{proof}

On the other hand the following proposition shows that $\cB(A)$ is
a g-alternative algebra up to the congruence relation $\simeq$.
\begin {Prop} For any $b_1, b_2, b_3 \in \cB(A), a\in A$, we have
$\mathbf{A}i\simeq 0$ for $i=1,\dots,11$.
\end {Prop}
\begin{proof}  Direct application of the statements \ref{bar_sac2}, \ref{bar_as2}, \ref{bar_aas2}
and \ref{bar_ap2}.
\end{proof}

\begin{coro}\label{asoci0} If $\asoci(A)=0$, then $\cB(A)$ is a
g-alternative algebra.
\end{coro}

\begin{Lem}\label{asoci10} Let $\{B_i\}_{i\in I}$ be the family of all g-alternative
algebras which have a derived action on $A$. The following
conditions are equivalent:
\begin{enumerate}
\item[(a)] $\asoci(A)=0$;

\item[(b)] $\asoci^1(A)=0$;

\item[(c)] every derived action of $B_i$ on $A$ is anticommutative, for
any $i\in I$ (i.e. $b_ia=-ab_i$, $b_i\in B_i$, $a\in A$) and
$\ann(A)=0$.
\end{enumerate}
\end{Lem}
\begin{proof} The implication (a)$\Rightarrow$(b) is obvious, since
$\asoci^1(A)\subseteq \asoci(A)$.

(b)$\Rightarrow$(c). We have $b_ia+ab_i\in \asoci^1(A)$; since
$\asoci^1(A)=0$,
 it follows that every derived action is anticommutative; in particular,
 since $A$ has a derived action on itself,  $A$ is anticommutative. From this it follows that the right and left
annulators of $A$ coincide and thus $\rann(A)=\lann(A)=\ann(A)$.
As we have noted in Section \ref{properties}, we have an inclusion
$\soci(A)\subseteq \asoci^1(A)$, from which we obtain that
 $\soci(A)=0$. Therefore
we obtain that $\ann(A)=\asoci^1(A)=0$.

(c)$\Rightarrow$(a). Since every action of $B_i$ on $A$ is
anticommutative, we have $\soci(A)=0$. Therefore for $x\in
\asoci(A)$ there exists a natural number $n$, such that for any
$a_1,\dots,a_n\in A$ we have $(\dots((xa_1)a_2)\dots a_n)=0$.
Since $\ann(A)=0$, it follows that $(\dots ((xa_1)a_2)\dots
a_{n-1})=0$. Thus applying analogous arguments we obtain that
$xa_1=0$ for any $a_1\in A$ and therefore $x=0$ since $\ann(A)=0$,
which ends the proof.
\end{proof}

\begin{Prop} \label{anti}If $A$ is anticommutative g-alternative algebra over a
field $F$ and $\ann (A)=0$, then $(a_1b)a_2=a_1(ba_2)$, for
any g-alternative algebra $B$ with derived action on $A$ and any
$a_1, a_2\in A, b\in B$.
\end{Prop}
\begin{proof} If $\ch F
\neq 2$, then by Lemma \ref{aas1} (iv) $A=0$, so the equality always
holds. Consider the case $\ch F=2$. Since $A$ is anticommutative,
by Lemma \ref{aas1} (ii), it is antiassociative as well, and
therefore, for any $a\in A$, we have the following equalities
\begin{multline*}
((a_1b)a_2)a=-(a_2(a_1b))a=(a_1(a_2b))a=-(a_1a)(a_2b)=a_2((a_1a)b)=\\-((a_1a)b)a_2=
-(a_1(ab)+a_1(ba)-(a_1b)a)a_2=-(-a(a_1b)+a_1(ba)-(a_1b)a)a_2=\\-((a_1b)a+a_1(ba)-(a_1b)a)a_2=
-(a_1(ba))a_2=a_1((ba)a_2)=-a_1((ba_2)a)=(a_1(ba_2))a, \\
(a_1(ba_2))a=-(a_1a)(ba_2)=(a_1(ba_2))a, \qquad  \qquad \qquad  \qquad \qquad
\end{multline*}
from which by the condition $\ann (A)=0$ follows the desired
equality.
\end{proof}

Note that the proof of this proposition doesn't follow from Lemma
\ref{aas1} (i) by taking there $x=a_1, y=b, z=a_2$, since in the
proof of $S_1^{\rm as}\sim0$ is involved $\soci(A)$, which
contains an element from $B$.

\begin{Prop}\label{ann0} Let $A$ be a g-alternative algebra with  $\ann(A)=0$. The following conditions are equivalent:
\begin{enumerate}
\item[(i)] for any g-alternative algebra $B$ with derived action on $A$
in $\galt$, we have $ba=-ab$ for any $a\in A, b\in B$;
\item[(ii)] $A$ is anticommutative.
\end{enumerate}
\end{Prop}
\begin{proof} Here, as in the previous proposition,  we need to
prove only the case, where $\ch F=2$. It is obvious that (i)
$\Rightarrow$ (ii), since A has a derived action on itself.

(ii) $\Rightarrow$ (i). For any
$a',a''\in A$ we have

$a''(a'(ab)+a'(ba))=a''(-a(a'b)+a'(ba))=a''(-a(a'b)+(a'b)a+(ba')a-b(a'a))=
a''(-a(a'b)-a(a'b)+(ba')a-(ba')a-(a'b)a+a'(ba))=-2a''(a(a'b))=0$.

Here we applied anticommutativity of $A$, Axiom~2, Proposition \ref{anti}
and Lemma \ref{aas1} (iii).
\end{proof}

\begin{Prop}\label{antig_ann0} If $A$ is an anticommutative g-alternative algebra and $\ann(A)=0$, then $\cB(A)$ is a g-alternative algebra.
\end{Prop}
\begin{proof} Apply Corollary \ref{asoci0}, Lemma \ref{asoci10} and Proposition \ref{ann0}.
\end{proof}
\section{Sufficient conditions for the existence of an actor in $\galt$}\label{existence_galt}

 It is obvious that the action of $\cB(A)$ on  $A$ satisfies identities II$_1$ and
 II$_2$; but this action in general is not a derived action. We begin
 with examples, which show that all other action identities fail in
general.

 \begin{ex} \label{ex_prod2} Here we consider the product algebra $A\times A$ of  Example \ref{ex_prod} and  we show
 that the identity II$_3$ is not true in general. We take in
 II$_3$:
 instead of $a$ the element $(a,a')\in A\times A, b_1=1, b_2=\lambda$ and obtain that in this case II$_3$
 is equivalent to the following equality \
$(0,a\lambda)=(0,2a\lambda)-(0,\lambda a)$.

 From this, since by assumption $a\lambda=-\lambda a$ and $\ch
 F\neq 2$, we obtain that $a\lambda=0$, which is not true in general.
 \end{ex}

 \begin{ex} Consider the same example of the algebra $A\times A$ as
 in Examples \ref{ex_prod} and \ref{ex_prod2}. We take in II$_4$ instead of $a$ the element $(a,a'), b_1=\lambda,
b_2=1$ and obtain \
$(0,\lambda a)=(0,2\lambda a)-(0,a\lambda)$.

 As in the previous example, this implies $a\lambda=0$,
 which is not true in general.
 \end{ex}

\begin{ex} Here we consider the example which shows that the
 identity I$_1$ is not always true.
 Let $A$ and $R$ be commutative, associative algebras over a
 field $F$ with characteristic 2,
 and $R$ has a derived action on $A$ in the category of commutative, associative
 algebras, i.e. together with the conditions given in Section \ref{def}. We
 have $ra=ar$, for any $a\in A$ and $r\in R$. Obviously this will be a derived action
 in the category of g-alternative algebras as well. Let
$\Lambda$ be a
 g-alternative algebra over the same field $F$, which acts on $A$
 in $\galt$, and there exists an element $a' \in A$, such that
 $a'\lambda$ is not a zero divisor in $A$. We shall show that
 the following equality is not true
\[(r\lambda)(aa')=((r\lambda)a)a'+(a(r\lambda))a'-a((r\lambda)a') \, ,\]
for any $a\in A$. Under the above assumptions on actions and algebras
the computations give the following:

\begin{itemize}

\item  $(r\lambda)(aa')=r(\lambda(aa'))+r((aa')\lambda)-(r(aa'))\lambda=
 r((\lambda a)a')+r((a\lambda)a')-r(a(\lambda a'))-r((a\lambda)a')+((ra)\lambda)a'$;

\item $((r\lambda)a)a'=(r(\lambda a))a'+(r(a\lambda))a'-((ra)\lambda)a'$;

 \item $(a(r\lambda))a'=-(r(a\lambda)a'$;

 \item $-a((r\lambda)a')=-a(r(\lambda
 a'))-a(r(a'\lambda))+a((ra')\lambda)$.
  \end{itemize}

 Thus the above equality is equivalent to the following one:
$a((ra')\lambda)=(ar)(a'\lambda)$.

Applying the fact that $A$ is commutative and  $\ch F=2$, we have
$a((ra')\lambda)=-(ra')(a\lambda)$; therefore we obtain
 $-(ra')(a\lambda)=(ar)(a'\lambda)$,
 which can not be true for any $a$, since in this case
 $a\lambda=0$ implies $ar=0$, because $a'\lambda$ is not a zero
 devisor by assumption.

 The cases of the identities I$_2$, I$_3$, I$_4 $ are considered in
 analogous ways.
 \end{ex}

 \begin{theo}\label{asoci0_derived} If $\asoci(A)=0$, then the action of $\cB(A)$ on
 $A$ is a derived action.
 \end{theo}
 \begin{proof} From the definition of the multiplication in $\cB(A)$ it is easy to see that the identities II$_1$ and
 II$_2$ of Section \ref{acting}
 always hold for the action of  $\cB(A)$ on $A$. For the identities
 $\mbox{I}_1-\mbox{I}_4$ and II$_3$ and II$_4$ we apply statements \ref{bar_aas1} and \ref{bar_as2},
 which show that under the condition of the theorem
$\cB(A)$ has a derived action on $A$. Note that the same can be
proved by the statements \ref{bar sac1}, \ref{bar_ap1},
\ref{bar_sac2} and \ref{bar_ap2}.
\end{proof}

\begin{coro}\label{asoci0_actor} If $\asoci(A)=0$, then $\cB(A)$ is an actor of
$A$.
\end{coro}
\begin{proof} Apply Corollaries \ref{actor_interest}, \ref{asoci0} and Theorem \ref{asoci0_derived}.
\end{proof}

\begin{coro} \label{anti_ann0} If $A$ is anticommutative g-alternative algebra over a field $F$ and
$\ann(A)=0$, then there exists an actor of $A$ and
$\act(A)=\cB(A)$.
\end{coro}
\begin{proof} If $\ch F
\neq 2$, then, by Lemma \ref{aas1} (iv), it follows that $A=0$, and
obviously $\act(A)=\cB(A)=0$. If $\ch F=2$, we apply Lemma
\ref{asoci10}, Propositions \ref{ann0} and \ref{antig_ann0},
Corollary \ref{asoci0_actor} and obtain the result.
\end{proof}

Here we give the construction of an $F$-algebra of
bimultiplications $\bim_{\galt}(A)$ of a g-alternative algebra $A$
over a field $F$. Below is used the notation $f*$ and $*f$ for the
$F$-linear maps $A\rightarrow A$; we will denote by $fa$ (resp.
$af$) the value $(f*)(a)$ (resp. $(*f)(a)$). This kind of notation
(similar to the one of the actions $b*a$ and $a*b$ in a category
of interest) makes simpler to write down the conditions for
bimultiplications; we will see that these conditions are simply
Axioms~$2_1$ and $2_2$ written for the four different ordered
triples. An element of $\bim_{\galt}(A)$ is a pair $f=(f*,*f)$ of
$F$-linear maps from $A$ to $A$,  which satisfies the following
conditions
\begin{equation}\label{pair}
\begin{split}
f (a_1 a_2)& = (f a_1) a_2 + (a_1 f)a_2 - a_1(f a_2) \, , \\
(a_1 a_2) f & = a_1 (a_2 f) + a_1 (f a_2) - (a_1 f) a_2  \, ,\\
(f a_1) a_2 & = f (a_1 a_2) + f (a_2 a_1) - (f a_2) a_1  \, , \\
a_1(a_2 f)&=(a_1 a_2)f + (a_2 a_1) f - a_2(a_1 f) \, .
\end{split}
\end{equation}
 The product of the elements $f=(f*,*f)$ and $f'=(f'*,*f')$ of
$\bim_{\galt}(A)$ is defined by
\[  ff'=(f*f'*, *f*f')\, ,    \]
here on the right side $f*f'*$ \ and \ $*f*f'$ are defined by
\begin{align*}
&(f*f'*)(a)=f(f'a)+f(af')-(fa)f' \, ,\\
&(*f*f')(a)=(af)f'+(fa)f'-f(af') \, .
\end{align*}
For the addition we have
\[f+f'=((f*)+f'*,*f+(*f')),\]
where
\begin{align*}
&((f*)+f'*)(a)=fa+f'a \, ,\\
&(*f+(*f'))(a)=af+af'.
\end{align*}

The product of two bimultiplications may not have the properties
\eqref{pair}. Therefore we include in $\bim_{\galt}(A)$ all the new
obtained pairs of $F$-linear maps. Note that different products
can give one and the same pairs of maps, i.e.
$(\varphi*,*\varphi)=(\varphi'*,*\varphi')$ if $\varphi a=\varphi'
a$ and $a \varphi=a \varphi'$, where $\varphi=(\varphi*,*\varphi)$
and $\varphi'=(\varphi'*,*\varphi')$ are certain combinations of
bimultiplications.

It is obvious that $\bim_{\galt}(A)$ is an $F$-algebra and it
is an object of $\galt_G$ in general. In the same way as
Corollary \ref{anti_ann0}, it can be proved
\begin{theo}\label{act_bim}
Let $A$ be an anticommutative g-alternative algebra with
$\ann(A)=0$, then there exists an actor of $A$ and
$\act(A)=\bim_{\galt}(A)$.

\end{theo}

\begin{coro}\label{bim_beta} Under the condition of Theorem \ref{act_bim} we have
$\bim_{\galt}(A)=\cB(A)$.
\end{coro}
\begin{proof} The conditions of Corollary \ref{anti_ann0} are fulfilled, from
which it follows that $\cB(A)$ is an actor of $A$. From Theorem
\ref{act_bim} and the universal property of an actor we obtain the desired
equality.
\end{proof}

 \section{The existence of an actor in $\alt$}\label{existence_alt}

 As we have noted in Section \ref{def}, by definition of general category of
 interest,
 we have $\alt_G=\galt_G$. Let $\alt$ be a category of alternative
 algebras
 over a field $F$ with $\ch F=2$. Then, by Proposition \ref{galt}, $\alt
\subset \galt$. In the categories $\alt$ and $\galt$ Axiom~2 is
the same, according to which the multiplication in the
construction of a universal strict general actor is defined.
Therefore, for any $A\in \alt$, the algebra $\cB_{\alt}(A)$,
constructed for the derived actions in $\alt$, is a subalgebra of
$\cB(A)$, constructed for the derived actions in $\galt$ for the
same alternative algebra $A$. Thus we have the inclusion of
algebras $\cB_{\alt}(A)\subseteq \cB(A)$.

 \begin{Prop}\label{ann0_beta} If $A$ is anticommutative
 g-alternative algebra
 and $\ann(A)=0$, then $\cB(A)$ is also anticommutative and
 $\ann(\cB(A))=0$.
 \end{Prop}
\begin{proof} If $\ch F
\neq 2$, then, by Lemma \ref{aas1} (iv), it follows that $A=0$, and
obviously $\cB(A)=\ann(\cB(A))=0$. If $\ch F=2$, then by
Proposition \ref{ann0} every derived action on $A$ is
 anticommutative; from
 this and the condition $\ann(A)=0$, by Lemma \ref{asoci10} it follows that $\asoci(A)=0$.
 Applying this fact, from Proposition \ref{bar_sac2}, we obtain that
 $\bar{S}_2^{\rm sac}=0$, which gives the equalities $a(b_1b_2)=-a(b_2b_1)$ and
 $(b_1b_2)a=-(b_2b_1)a$, for any $a\in A$ and $b_1,b_2 \in \cB(A)$, which
 by the construction of $\cB(A)$, proves that $\cB(A)$ is
 anticommutative. Now we shall prove that $\ann(\cB(A))=0$.
 Suppose $b_1b=0$ for any $b\in \cB(A)$, which means that
 $(b_1b)a=0$, for any $a\in A$. By the definition of multiplication in $\cB(A)$ and an action of
 $\cB(A)$ on $A$ we have
\[(b_1b)a=b_1(ba)+b_1(ab)-(b_1a)b=0 \, .\]
Since the action of $\cB(A)$ is anticommutative,  we
 obtain that $(b_1a)b=0$, for any $b\in \cB(A)$,  and in particular,
 for $b=\aaa'$, where $a'$ is any element from $A$. This gives $(b_1a)a'=0$,
 for any $a'$, which means that $b_1a$ is an annulator
 in $A$, therefore $b_1a=0$, for any $a\in A$, which, by
 construction of $\cB(A)$, implies that $b_1=0$.
\end{proof}

Note that in the proof of Proposition \ref{ann0_beta} we could take $b=\aaa$
in $b_1b=0$, for any $a\in A$. In this case we should prove that
$(b_1\aaa)a'=(b_1a)a'$, for any $a'\in A$. For this we would have
to apply the fact that $\cB(A)$ has a derived action on $A$ in
$\galt$. From this we would obtain that $b_1\aaa$ and $b_1a$ are
equal in $\cB(A)$, i.e. $b_1\aaa=cl(b_1a)$, which will imply that
$(b_1a)a'=0$. From this, since $\ann(A)=0$, it would follow that
$b_1a=0$, and therefore $b_1=0$. As we see this proof is longer
than we have presented.

It is easy to see, and it is noted in \cite {Sch,ZSSS},
that any commutative or anticommutative algebra satisfies the
flexible law E$_1$. Therefore, from this proposition and
Proposition \ref{galt} (i), it follows that, under the conditions of
Proposition \ref{ann0_beta}, $\cB(A)$ is an alternative algebra. The following
corollary proves that $\cB(A)$ is an associative algebra as well.

\begin{coro}\label{beta_alt} If $A$ is anticommutative
g-alternative algebra
 with $\ann(A)=0$, then $\cB(A)$ is an associative algebra,
 in particular, $\cB(A)\in \alt$.
 \end{coro}
\begin{proof} Apply Proposition \ref{ann0_beta} and Lemma \ref{aas1} (iii).
\end{proof}

\begin{theo} \label{act_alt} If $A$ is anticommutative alternative algebra over a
field $F$ with $\ann(A)=0$, then there exists an actor of
$A$ in $\alt$ and
 $\act(A)=\cB_{\alt}(A)=\cB(A)$.
 \end{theo}
\begin{proof} By Proposition \ref{galt}, for any field $F$, we have $\alt \subseteq
\galt$. If $\ch F \neq 2$, then, by Lemma \ref{aas1} (iv), it follows
that $A=0$, and obviously $\act(A)=\cB(A)=0$. Suppose $\ch F=2$.
Since $A$ is a g-alternative algebra,
 we can apply Corollary \ref{beta_alt}, and therefore
$\cB(A)$ is an alternative algebra. By Corollary \ref{anti_ann0} the action of
$\cB(A)$ on $A$
 is a derived action in $\galt$.
 We shall prove that the action of $\cB(A)$ on $A$ satisfies
 the conditions III$_1$ and III$_2$ of Section \ref{acting}, for any $a\in A$
 and $b\in \cB(A)$. By Proposition \ref{ann0}, every derived action on $A$ is anticommutative, and by Lemma \ref{asoci10}, under the conditions of the
 theorem we obtain $\asoci(A)=0$; therefore applying Lemma
 \ref{bar_aas1} and Proposition \ref{bar_as2} we obtain that $\bar{S}_1^{\rm as}=0$ and
 $\bar{S}_2^{\rm as}=0$, which imply respectively the equalities III$_1$ and
 III$_2$. $\cB(A)$ is an actor of $A$ in $\galt$, therefore by the definition of an actor we obtain, that $\cB(A)$
 is an actor of $A$ in $\alt$. Now the result follows from
 Theorem \ref{actor_semi}.
\end{proof}

 The construction of an $F$-algebra of bimultiplications
 $\bim_{\alt}(A)$ of an alternative algebra $A$ over a field $F$ is analogous
 to $\bim_{\galt}(A)$, where in addition we require that  pairs $f=(f*,*f)$ of $F$-linear
maps from $A$ to $A$ together with conditions \eqref{pair} satisfy the
following two conditions:
\begin{align*}
a(fa)=(af)a \, ,\\
f(af)=(fa)f \, .
\end{align*}
The statements analogous to Theorem \ref{act_bim} and Corollary
\ref{bim_beta} take place for alternative algebras. Obviously,
every commutative associative algebra over a field with
characteristic 2 is anticommutative g-alternative algebra, and
every anticommutative g-alternative algebra $A$ over the same kind
field with $\ann(A)=0$ is associative and commutative (apply Lemma
\ref{aas1} (ii)). Let $\bim(A)$ denote the algebra of bimultiplications
of an associative algebra $A$ \cite{Ho}. It is proved in
\cite{CDL2} that if $A$ is any associative algebra  with $\ann(A)=0$, then there
exists an actor of $A$ in the category of associative algebras and
$\act(A)=\bim(A)$. The analogous result we have in the category of
commutative associative algebras, under the same condition an
actor exists and this is the commutative algebra of
multiplications $M(A)$ \cite{LS} (or multipliers \cite{LL}) of
$A$, which is defined as an algebra of $F$-linear maps $f \colon A \to A$
with $f(aa')=f(a)a'$, for any $a,a'\in A$. Let $A$ be an
associative commutative algebra over a field $F$ with
characteristic 2 and $\ann(A)=0$. The equalities $\bar{S}_1^{\rm
as}=0$ and
 $\bar{S}_2^{\rm as}=0$ in the
proof of Theorem \ref{act_alt} imply that the action of $\cB(A)$ on $A$ is a
derived action in the category of associative algebras. At the
same time we know that $\cB(A)$  is anticommutative and its action
on $A$ is anticommutative as well, therefore $\cB(A)$ is
commutative and the action on $A$ is commutative too. Therefore,
in the same way as in the proof of Theorem \ref{act_alt}, we conclude that
$\cB(A)$ is an actor of $A$ in the categories of associative and
commutative associative algebras. From the universal property of
an actor we obtain, that if $A$ is commutative associative algebra
over a field $F$ with characteristic 2 and $\ann(A)=0$, then
$\cB(A)=\bim(A)=M(A)$.

\section*{Acknowledgments}
The authors were supported by MICINN (Spain), Grant MTM 2009-14464-C02 (European FEDER support included) and project
Ingenio Mathematica (i-MATH) No. CSD2006-00032 (Consolider Ingenio
2010) and by Xunta de Galicia, Grant PGIDITI06PXIB371128PR. The
second author is grateful to Santiago de Compostela and Vigo
Universities and to Georgian National Science Foundation, Ref. ST06/3-004, for financial supports.


\begin{thebibliography}{99}
\bibitem{BJK}  F. Borceux,  G. Janelidze, G. M. Kelly,  Internal
object actions, Comment. Math. Univ. Carolinae 46 (2) (2005)
 235--255.

\bibitem{BJK1} F. Borceux, G. Janelidze,  G.M. Kelly, On the
representability of actions in a semi-abelian category,
Theory Appl. Categ. 14 (11) (2005) 244--286.

\bibitem{BB}  F. Borceux, D. Bourn, Split extension classifier
and centrality, in: Contemp. Math., vol. 431,
Categories in Algebra, Geometry and Mathematical Physics, 2007, pp. 85--104.


\bibitem{BBJ}  F. Borceux, D. Bourn, P. Johnstone,   Initial
normal covers in bi-Heyting toposes, Arch. Math. (Brno),
42 (4) (2006) 335--356.

\bibitem{Bou}  D. Bourn,  Action groupoid in protomodular
categories, Theory Appl. Categ. 16 (2) (2006) 46--58.


\bibitem{BJ}  D. Bourn,  G.  Janelidze,    Protomodularity,
descent, and semidirect products, Theory Appl. Categ. 4 (2) (1998) 37--46.

\bibitem{CDL1}  J.M. Casas, T. Datuashvili,  M. Ladra,  Actors in  categories of
interest,  arXiv:math.CT/0702574v2 25.02.2007.

\bibitem{CDL2} J.M. Casas, T. Datuashvili,  M. Ladra, Universal strict general actors and
actors in categories of interest,   Appl. Categ.
Structures (2009), DOI 10.1007/s10485-008-9166-z.

\bibitem{CDL3} J.M. Casas, T. Datuashvili,  M. Ladra,
Actor of a precrossed module, to appear in  Comm. Algebra
(2009).

\bibitem{Hig} J. Higgins,  Groups with multiple operators, Proc. London Math.
Soc. (3) 6 (1956) 366--416.

\bibitem{Ho} G. Hochschild, Cohomology and representations of associative
 algebras, Duke Math. J. 14 (1947) 921--948.

\bibitem{Kur}  A.G. Kurosh,  Lectures on general algebra, Translated from the Russian edition (Moscow, 1960) by K.A. Hirsch, Chelsea Publishing Co., New York, 1963.

\bibitem{LL} R. Lavendhomme, Th. Lucas,  On modules  and crossed modules, J. Algebra 179 (3) (1996) 936--963.

\bibitem{LS}  S. Lichtenbaum, M. Schlessinger,  The
cotangent complex of a morphism, Trans. Amer. Math. Soc.
128 (1967) 41--70.

\bibitem{Lo} J.-L. Loday,  Une version non commutative des
alg\`{e}bres de Lie: les alg\`{e}bres de Leibniz, Enseign.
Math. (2) 39 (3-4) (1993) 269--293.

\bibitem{Lu} A.S.-T. Lue,   Non-abelian cohomology of associative algebras,
Quart. J. Math. Oxford (2) 19 (1968) 159--180.

\bibitem{Mc} S. Mac Lane,  Extensions and obstructions for rings, Illinois J. Math. 2 (1958) 316--345.

\bibitem{No} K. Norrie,   Actions and automorphisms of crossed
modules, Bull. Soc. Math. France 118 (2) (1990) 129--146.

\bibitem{Orz} G. Orzech,   Obstruction theory in algebraic categories
I and II, J. Pure Appl. Algebra  2 (1972) 287--314 and
315--340.

\bibitem{Por} T. Porter,   Extensions, crossed modules and internal categories in
categories of groups with operations, Proc. Edinburgh Math.
Soc. (2) 30 (3) (1987) 373--381.

\bibitem{Sch} R.D. Schafer,  An introduction to
 nonassociative algebras, Pure and Applied Mathematics, vol. 22,  Academic Press, New York, London,  1966.

 \bibitem{ZSSS} K.A. Zhevlakov, A.M. Slin'ko, I.R. Shestakov,
 A.I. Shirshov,  Rings That Are Nearly Associative, Pure and Applied
 Mathematics, Academic Press, New York, London, 1982.
\end{thebibliography}
\end{document}